\documentclass{amsart}
\usepackage{url}
\usepackage{hyperref}
\hypersetup{colorlinks=true, citecolor=blue}
\providecommand{\texorpdfstring}[2]{#1}
\usepackage[lite]{amsrefs}
\usepackage{amssymb}
\usepackage{mathrsfs}
\usepackage[shortlabels]{enumitem}
\usepackage[all,cmtip]{xy}

\newcommand{\tensor}{\otimes}
\newcommand{\isom}{\cong}

\newcommand{\C}{\mathbb{C}}

\newcommand{\Q}{\mathbb{Q}}
\newcommand{\Z}{\mathbb{Z}}
\renewcommand{\P}{\mathbb{P}}

\newcommand{\M}{\mathcal{M}}
\newcommand{\D}{\mathcal{D}}
\newcommand{\F}{\mathcal{F}}
\newcommand{\G}{\mathcal{G}}
\renewcommand{\H}{\mathcal{H}}

\renewcommand{\O}{\mathcal{O}}
\newcommand{\DR}{\mathrm{DR}}

\renewcommand{\L}{\mathcal{L}}
\newcommand{\IC}{\mathcal{IC}}
\newcommand{\w}{\omega}

\newcommand{\supp}{\mathrm{supp\,}}
\newcommand{\gr}{\mathrm{gr}}

\numberwithin{equation}{section}

\theoremstyle{plain}
\newtheorem{thm}[equation]{Theorem}
\newtheorem{cor}[equation]{Corollary}
\newtheorem{lem}[equation]{Lemma}
\newtheorem{prop}[equation]{Proposition}
\newtheorem{conj}[equation]{Conjecture}

\theoremstyle{definition}
\newtheorem{defn}[equation]{Definition}

\theoremstyle{remark}
\newtheorem{rmk}[equation]{Remark}
\newtheorem{ex}[equation]{Example}

\begin{document}

\title[Viehweg hyperbolicity for Whitney equisingular families]{Viehweg hyperbolicity for Whitney equisingular families with Gorenstein rational singularities}
\author{Sung Gi Park}
\address{Department of Mathematics, Harvard University, 1 Oxford Street, Cambridge, MA 02138, USA}
\email{sgpark@math.harvard.edu}

\date{\today}

\begin{abstract}
We prove the analogue of Viehweg's hyperbolicity conjecture for Whitney equisingular families of projective varieties with Gorenstein rational singularities whose geometric generic fiber has a good minimal model. Namely, for such families with maximal variation, the base spaces are of log general type. The main new ingredient is the use of intersection complexes as Hodge modules in the construction of logarithmic Higgs sheaves by Viehweg-Zuo and Popa-Schnell. This construction suggests an equisingular stratification of the moduli space of varieties of general type, with each stratum being hyperbolic, and our result is a first step in this direction.
\end{abstract}

\maketitle

\tableofcontents

\section{Introduction}
\label{sec:intro}

Throughout the paper, a variety is a reduced connected separated scheme of finite type over \(\C\).

The goal of this paper is to establish Viehweg's hyperbolicity conjecture for families of singular varieties satisfying certain equisingularity conditions. Viehweg originally conjectured that the base of a family of canonically polarized complex manifolds of maximal variation is of log general type; this vastly generalizes Shafarevich's hyperbolicity conjecture on a non-isotrivial family of curves over a one dimensional base. Viehweg's statement was proven by an exhaustive list of cumulative works, of which we name a few: Viehweg-Zuo \cite{VZ01} when the base is a curve, Kebekus-Kov\' acs \cites{KK08, KK10} when the base is a surface or a threefold, and Campana-P\u aun \cite{CP19} in full generality. Popa-Schnell \cite{PS17}*{Theorem A} later generalized it to families of varieties admitting good minimal models:

\begin{thm}[Viehweg's Hyperbolicity Conjecture]
\label{thm: Viehweg's hyperbolicity}
Let \(V\) be a smooth quasi-projective variety and \(f:U\to V\) be a smooth family of projective varieties. If the general fiber of \(f\) has a good minimal model and \(f\) has maximal variation, then \(V\) is of log general type.
\end{thm}

Allowing singularities that appear in the minimal model program, it is natural to ask for an analogue of Viehweg's conjecture for families of singular varieties. For instance, we start by asking whether the statement in Theorem \ref{thm: Viehweg's hyperbolicity} holds for a flat projective family \(f\) with canonical singularities. However, it is immediately clear that this fails; indeed, a Lefschetz pencil of hypersurfaces of sufficiently large degree is a non-isotrivial canonically polarized family of varieties with at most one ordinary double point - the mildest of higher dimensional singularities. To remedy this, we add an equisingularity condition for families of varieties, namely Whitney equisingularity, inspired by Teissier \cite{Teissier75}*{Section 1} who defined an equivalent notion in the case of isolated hypersurface singularities.

\begin{defn}
\label{defn: Whitney equisingularity}
Let \(f:U\to V\) be a morphism of varieties, with \(V\) smooth. We say \(f\)  is \textit{Whitney equisingular} if there exists a Whitney stratification \(U=\coprod_{\alpha\in S}U_\alpha\) such that \(f|_{U_\alpha}:U_\alpha \to V\) is smooth for all \(\alpha\in S\).
\end{defn}

Historically, the Whitney equisingularity condition appears in the context of differential topology: Thom's first isotopy lemma \cite{Thom69} proves that a proper Whitney equisingular morphism is a topological locally trivial fibration. See Mather \cite{Mather12}*{Proposition 11.1} for a modern treatment.

Furthermore, for families of isolated hypersurface singularities, Whitney equisingular morphisms were studied in depth by Teissier \cite{Teissier73} and Brian\c con-Speder \cite{BS76} using a numerical invariant called the Milnor sequence \(\mu^{(*)}\). For a germ \((X_0,x_0)\subset (\C^{n+1},0)\) of an isolated hypersurface singularity, Teissier introduced the Milnor sequence
\[
\mu^{(*)}_{x_0}(X_0)=\left(\mu^{(n+1)}_{x_0}(X_0),\dots, \mu^{(i)}_{x_0}(X_0),\dots, \mu^{(0)}_{x_0}(X_0)\right)
\]
where \(\mu^{(i)}_{x_0}(X_0)\) is the Milnor number of the intersection \((X_0\cap H, x_0)\) with a general plane \(H\) of dimension \(i\) through \(x_0\). In particular, we have \(\mu_{x_0}^{(n+1)}(X_0)=\mu_{x_0}(X_0)\) the ordinary Milnor number, \(\mu_{x_0}^{(1)}(X_0)=m_{x_0}(X_0)-1\) where \(m_{x_0}(X_0)\) is the multiplicity of \(X_0\) at \(x_0\), and \(\mu_{x_0}^{(0)}(X_0)=1\) by convention. Teissier proved that the germ of a flat deformation of isolated hypersurface singularities is Whitney equisingular if the Milnor sequence is constant, and Brian\c con-Speder proved the converse statement.

In our main result here, we verify the analogue of Viehweg's hyperbolicity conjecture for Whitney equisingular families of projective varieties with Gorenstein rational (or equivalently, Gorenstein canonical) singularities.

\begin{thm}
\label{thm: Viehweg hyperbolicity for Gorenstein rational singularities}
Let \(V\) be a smooth quasi-projective variety and \(f:U\to V\) be a flat Whitney equisingular family of projective varieties with Gorenstein rational singularities. If the general fiber of \(f\) has a good minimal model and \(f\) has maximal variation, then \(V\) is of log general type.
\end{thm}

Recently, Kov\'acs-Taji \cite{KT21}*{Theorem 1.3} obtained a version of Viehweg hyperbolicity for Gorenstein rational families under some positivity assumption on the pushforwards of pluri-canonical bundles. At the moment, we don't know how to apply this in the present context; this is an interesting problem.

For isolated hypersurface singularities, we can be more precise due to the numerical characterization of Whitney equisingularity by Teissier, Brian\c con and Speder, described in Theorem \ref{thm: Whitney equisingular, constant Milnor sequence}. Hence, we have the following:

\begin{cor}
\label{cor: Viehweg hyperbolicity for isolated hypersurface singularities}
Let \(f:U\to V\) be a flat family of projective varieties with at worst isolated rational hypersurface singularities. Assume that each nonzero Milnor sequence appears a fixed number of times on every fiber. If the general fiber of \(f\) has a good minimal model and \(f\) has maximal variation, then \(V\) is of log general type.
\end{cor}

Notice that the existence of a simultaneous resolution of \(f\) would imply Theorem \ref{thm: Viehweg hyperbolicity for Gorenstein rational singularities} and Corollary \ref{cor: Viehweg hyperbolicity for isolated hypersurface singularities}, due to Theorem \ref{thm: Viehweg's hyperbolicity}. In fact, this is the case for families of surfaces. For example, let \(f:U\to V\) be a flat family of projective surfaces with Du Val singularities and non-negative Kodaira dimension. Then, \(f\) is Whitney equisingular if and only if each type of \(ADE\)-singularities appears a fixed number of times on every fiber, in which case \(f\) admits a simultaneous resolution of singularities. Therefore, if additionally \(f\) has maximal variation, then \(V\) is of log general type. See Example \ref{ex: surface DV} for details.

As soon as we move to families of higher dimensional varieties (threefolds or higher), we are very far from knowing the existence of a simultaneous resolution of singularities. Teissier \cite{Teissier80}*{Questions 4.10.1} only predicted its existence for the analytic germ of a flat deformation of an isolated hypersurface singularity with constant Milnor sequence. Laufer \cites{Laufer83, Laufer87} answered this affirmatively in the case of surface singularities, but for threefolds or higher dimensional varieties, Teissier's question remains wide open. See Teissier's and Laufer's papers for precise statements on simultaneous resolutions.

It is already interesting to look at a flat family \(f:U\to V\) of projective threefolds with isolated compound Du Val (i.e. cDV) singularities and non-negative Kodaira dimension. Recall that the cDV singularities are rational hypersurface singularities classified by the types (\(cA,cD,cE\)). It is explained in Example \ref{ex: threefold cDV} that \(f\) is Whitney equisingular if each type of \(cA,cD,cE\)-singularities with a fixed embedded topological type appears a fixed number of times on every fiber.

For instance, let \(\bar f:X\to \P^1\) be a birationally non-isotrivial flat family of projective threefolds whose general fiber has non-negative Kodaira dimension and has exactly one \(cA_5\)-singularity of a fixed embedded topological type (i.e. a fixed Milnor number \cite{LR76}). In this case, we say a fiber is equisingular to the general fiber if it has one \(cA_5\)-singularity with the same embedded topological type. Then there exist at least three fibers, that are not equisingular to the general fiber. In Proposition \ref{prop: boundary example}, we give an example with exactly three non-equisingular fibers, which suggests that Theorem \ref{thm: Viehweg hyperbolicity for Gorenstein rational singularities} and Corollary \ref{cor: Viehweg hyperbolicity for isolated hypersurface singularities} are in some sense optimal.

\subsubsection*{\textnormal{\textbf{Interpretation via moduli theory.}}}

Let \(\M_h\) be the Deligne-Mumford moduli stack of canonically polarized manifolds with Hilbert polynomial \(h\) and \(M_h\) be its coarse moduli space. Viehweg's original conjecture is equivalent to the hyperbolicity of \(\M_h\) in a birational geometric sense: for every generically finite map \(V\to M_h\) induced by a family \(f:U\to V\) of canonically polarized complex manifolds, the base \(V\) is of log general type. Imposing Whitney equisingularity, we explain how this hyperbolicity property generalizes to the moduli space of varieties of general type (described for instance in Koll\'ar's book \cite{Kollar21}), or more precisely, the moduli space of KSB-stable varieties.

Due to the well-known hyperbolicity of \(\M_{g,n}\), it is easy to check that the moduli space \(\overline \M_g\) of stable curves of genus \(g\ge 2\) admits a stratification into hyperbolic strata defined by the equisingularity:
\[
\overline \M_g=\coprod_{\delta}\M_g^\delta
\]
where \(\M_g^\delta\) is the moduli space of stable curves of genus \(g\) with \(\delta\) nodes. We expect the same to hold for the moduli space of KSB-stable varieties. Corollary \ref{cor: Viehweg hyperbolicity for isolated hypersurface singularities} makes a first step in this direction, in the case of KSB-stable varieties with isolated rational hypersurface singularities as follows. We fix a multiset (i.e. a set of elements counted with multiplicities) of Milnor sequences \(\{\mu^{(*)}_s\}_{s\in S}\) with \(r:=|S|\) elements. Let \(\M_{n,v}^{\{\mu^{(*)}_s\}}\) be the moduli stack of canonically polarized varieties of fixed dimension \(n\) and volume \(v\) with \(r\) isolated rational hypersurface singularities, whose multiset of Milnor sequences is \(\{\mu^{(*)}_s\}\). By the upper semicontinuity of the Milnor numbers, \(\M_{n,v}^{\{\mu^{(*)}_s\}}\) is a locally closed substack of the moduli stack \(\M_{n,v}\) of KSB-stable varieties of dimension \(n\) and volume \(v\). Then Corollary \ref{cor: Viehweg hyperbolicity for isolated hypersurface singularities} implies that \(\M_{n,v}^{\{\mu^{(*)}_s\}}\) is hyperbolic (if nonempty) in a birational geometric sense. See Section \ref{subsec: numerical criteria} for details. 

In the hope that the above hyperbolicity property extends to every Whitney equisingular stratum of \(\M_{n,v}\), we propose the following conjecture for KSB-stable families, analogous to Theorem \ref{thm: Viehweg hyperbolicity for Gorenstein rational singularities}.

\begin{conj}
Let \(f:U\to V\) be a KSB-stable Whitney equisingular family. If \(f\) has maximal variation, then \(V\) is of log general type.
\end{conj}

\begin{rmk}
On a related note, it is natural to consider an Arakelov-type inequality for \(f\). For instance, Kov\'acs \cite{Kovacs02}*{Section 7} established this inequality for families of canonically polarized varieties with Gorenstein rational singularities over curves, which admit a simultaneous resolution of singularities. As a byproduct, Kov\'acs obtained \textit{weak boundedness} for such families. We expect a similar boundedness property to hold for KSB-stable Whitney equisingular families.
\end{rmk}

\subsubsection*{\textnormal{\textbf{What is new.}}}

Given a flat family \(f:U\to V\), let \(Y\) be a smooth compactification of the base \(V\) with boundary a simple normal crossing divisor \(D\). In the literature, the proof of Viehweg's hyperbolicity conjecture for families of smooth varieties appeals to the construction of a Viehweg-Zuo sheaf, i.e. a big line bundle contained in some large tensor power of the logarithmic cotangent bundle:
\[
A\subset \Omega_Y(\log D)^{\tensor N},
\]
obtained from a logarithmic Higgs sheaf on \(Y\) with poles along \(D\). Popa-Schnell \cite{PS17} refined the construction, originally due to Viehweg-Zuo \cites{VZ01, VZ02}, using the theory of polarizable Hodge modules.

The main new input of this paper is the use of \(\IC_U^H\), the polarizable Hodge module associated to the intersection complex of \(U\), in the construction of a Viehweg-Zuo sheaf. The major difference, hence difficulty, stems from \(U\) having singularities. For this to work through, we need two additional conditions.

One is Whitney equisingularity of \(f\). From the pushforward of a natural object associated to \(\IC_U^H\), we aim to construct a logarithmic Higgs sheaf on \(Y\) with poles along \(D\). In other words, its support restricted to the cotangent bundle of \(V\) should be the zero section. Hence, we require the characteristic variety of the intersection complex to be, in some sense, \textit{transversal} to \(f\). Therefore, we ask for \(f\) to be Whitney equisingular for this purpose, as explained in Section \ref{sec: intersection complex and Whitney equisingular morphism}.

The other is requiring the fibers to have Gorenstein rational singularities. In the construction of a Viehweg-Zuo sheaf from a logarithmic Higgs sheaf, the positivity of its first nonzero graded piece is crucial. Therefore, we need some control over the first nonzero term of the Hodge filtration of \(\IC_U^H\); when \(U\) has Gorenstein rational singularities, this is equal to the dualizing line bundle \(\w_U\). In a general setting, the absence of this property is a difficulty that we have not been able to overcome. See Section \ref{sec: Hodge module} for the details, along with the basics on Hodge modules.

Carefully putting these pieces together in the construction of a Viehweg-Zuo sheaf by Popa and Schnell, we obtain our main theorem \ref{thm: Viehweg hyperbolicity for Gorenstein rational singularities}.

\subsubsection*{\textnormal{\textbf{Acknowledgements.}}}

I would like to thank my advisor, Mihnea Popa, for suggesting this problem and for the helpful conversations. I also thank Joe Harris for the helpful discussions. I am supported by the 17th Kwanjeong Study Abroad Scholarship.

\section{Technicalities on Hodge modules}
\label{sec: Hodge module}

In this section, we mostly review the background on filtered D-modules and Saito's theory of Hodge modules, in order to integrate the intersection complex into the construction of a Viehweg-Zuo sheaf.

We follow notation and terminology in Saito \cites{Saito88, Saito90}. Let \(X\) be a smooth algebraic variety, and \(\D_X\) be the sheaf of differential operators on \(X\). Recall that \(\D_X\) is endowed with a canonical filtration \(F\), via the order of differential operators. Let \(MF(\D_X)\) be the category of filtered (right) \(\D_X\)-modules and let \(D^bF(\D_X)\) be the associated bounded derived category of filtered (right) \(\D_X\)-modules. Unless otherwise stated, we consider right D-modules in this section. In fact, the left-right transformation of filtered D-modules, i.e.
\[
\w_X\tensor (\bullet): MF(\D_X)_{\textrm{left}}\to MF(\D_X)_{\textrm{right}}, \quad
(N,F_\bullet)\mapsto (\omega_X\tensor N,\omega_X\tensor F_{\bullet+\dim X}),\]
recovers the theory in terms of left D-modules.

In Section \ref{sec: derived pushforward and associated graded}, we explain the basic functors of filtered D-modules. In Section \ref{sec: cyclic cover and hodge module}, the morphism of complexes induced by a cyclic cover (Viehweg-Zuo \cites{VZ01, VZ02} and Popa-Schnell \cite{PS17}) is stated in terms of Hodge modules. In Section \ref{sec: intersection complex}, we explain Saito's Decomposition Theorem and use it to study some important properties of the intersection complex as a Hodge module.

\subsection{Derived pushforward and associated graded functor}
\label{sec: derived pushforward and associated graded}

In order to achieve the construction of Higgs sheaves by Viehweg-Zuo in the generalized setting, not necessarily for families of smooth varieties, we discuss the derived pushforward and the associated graded functor for a complex of filtered D-modules. These standard functors appear in Saito's theory \cites{Saito88, Saito90}, but we make the statements more precise and formal, especially regarding the commutativity of the proper pushforward functor and the associated graded functor in Proposition \ref{prop: graded pushforward commuting diagram}. This is crucial when we are to incorporate the intersection complex later in the paper, whose associated graded complex generalizes relative differentials in the smooth case.

A priori, for a proper morphism \(f:X\to Y\) of smooth varieties, Saito constructed the derived pushforward \(f_+:D^bF(\D_X)\to D^bF(\D_Y)\) of derived categories of filtered D-modules as:
\[
f_+M^\cdot:=Rf_*(M^\cdot\tensor^L_{\D_X}\D_{X\to Y}), \quad \D_{X\to Y}:=\O_X\tensor_{f^{-1}\O_Y}f^{-1}\D_Y,
\]
where the \textit{transfer module} \(\D_{X\to Y}\) is a \((\D_X,f^{-1}\D_Y)\)-bimodule. More precisely, for a filtered right \(\D_X\)-module \((M,F)\) with the associated left \(\D_X\)-module \((N,F)\), we have the filtered Spencer resolution of filtered right \(\D_X\)-modules (c.f. \cite{HTT}*{Lemma 1.5.27} and \cite{Saito88}*{Lemme 2.1.6}):
\[
0\to (N\tensor_{\O_X}\Omega_X^0)\tensor_{\O_X}\D_X\to\cdots\to (N\tensor_{\O_X}\Omega_X^n)\tensor_{\O_X}\D_X\to M\to 0
\]
where
\[
F_k\left((N\tensor_{\O_X}\Omega_X^i)\tensor_{\O_X}\D_X\right):=\sum_{j} (F_{k+i-j}N\tensor_{\O_X}\Omega^i_X)\tensor_{\O_X} F_j\D_X.
\]
Hence, the derived tensor product \(M\tensor^L_{\D_X}\D_{X\to Y}\) is naturally constructed as the following complex:
\begin{equation}
\label{eqn: relative Spencer resolution}
0\to (N\tensor_{\O_X}\Omega_X^0)\tensor_{f^{-1}\O_Y}f^{-1}\D_Y\to\cdots\to (N\tensor_{\O_X}\Omega_X^n)\tensor_{f^{-1}\O_Y}f^{-1}\D_Y\to 0
\end{equation}
endowed with the filtration given by
\[
F_k\left((N\tensor_{\O_X}\Omega_X^i)\tensor_{f^{-1}\O_Y}f^{-1}\D_Y\right):=\sum_{j} (F_{k+i-j}N\tensor_{\O_X}\Omega^i_X)\tensor_{f^{-1}\O_Y} f^{-1}F_j\D_Y.
\]
Taking the filtered derived pushforward of \eqref{eqn: relative Spencer resolution} via the filtered version of the Godement resolution, we obtain \(f_+M\in D^bF(\D_Y)\). See \cite{Saito88}*{Section 2.3} for details.

\begin{rmk}
\label{rmk: boundary morphism}
In local coordinates \(\left\{y_j,\partial_j\right\}\) on \(Y\), the boundary morphisms of \eqref{eqn: relative Spencer resolution},
\[
d:(N\tensor_{\O_X}\Omega_X^i)\tensor_{f^{-1}\O_Y}f^{-1}\D_Y\to(N\tensor_{\O_X}\Omega_X^{i+1})\tensor_{f^{-1}\O_Y}f^{-1}\D_Y,
\]
are given by
\[
d(\omega\tensor P)=d\omega\tensor P+\sum_j f^*(dy_j)\wedge\omega\tensor \partial_j P,
\]
where \(\omega\) (resp. \(P\)) is a local section of \(N\tensor_{\O_X}\Omega_X^i\) (resp. \(\D_Y\)) and \(d\omega\) is its differential. This is easily obtained from the description of the Spencer resolution (c.f. \cite{HTT}*{Lemma 1.5.27}).
\end{rmk}

On a different note, for an arbitrary filtered sheaf \((M,F)\) on \(X\), not necessarily a filtered \(\D_X\)-module, there exists an associated graded sheaf
\[
\gr^FM:=\bigoplus_k \gr^F_kM.
\]
In particular, \(\gr^F\D_X=Sym^\bullet T_X\). Therefore, the associated graded complex of \eqref{eqn: relative Spencer resolution} is a complex of graded \(f^*\gr^F\D_Y:=\O_X\tensor_{f^{-1}\O_Y}f^{-1}\gr^F\D_Y\)-modules, and there exists a functor
\[
\gr^F\DR_{X/Y}:D^bF(\D_X)\to D^bG(f^*\gr^F\D_Y), \quad M^\cdot\mapsto \gr^F(M^\cdot\tensor^L_{\D_X}\D_{X\to Y}),
\]
where \(D^bG(f^*\gr^F\D_Y)\) is the bounded derived category of graded \(f^*\gr^F\D_Y\)-modules. This functor commutes with the pushforward functors, described next.

\begin{prop}
\label{prop: graded pushforward commuting diagram}
For a proper morphism \(\phi:Z\to X\) of smooth varieties over \(Y\), we have the following diagram:
\begin{equation*}
\label{diagram: graded pushforward commuting diagram}
\xymatrix@C=10ex{
{D^bF(\D_Z)} \ar[r]^-{\gr^F\DR_{Z/Y}} \ar[d]_{\phi_+} & {D^bG((f\circ\phi)^*(\gr^F\D_Y))} \ar[d]_{\phi_*} \\
{D^bF(\D_X)} \ar[r]^-{\gr^F\DR_{X/Y}} & {D^bG(f^*\gr^F\D_Y)}
}
\end{equation*}
which commutes up to natural isomorphism. Here, \(\phi_*\) is the derived pushforward functor of graded modules.
\end{prop}

In other words, we have a natural isomorphism of functors
\[
\phi_*\circ \gr^F\DR_{Z/Y}\isom\gr^F\DR_{X/Y}\circ \phi_+.
\]

\begin{proof}
This is an analogue of \cite{Saito88}*{2.3.7}, which is immediate from the construction of each functor.
\end{proof}

\subsection{Cyclic covers and Hodge modules}
\label{sec: cyclic cover and hodge module}

Cyclic covers appear as a crucial ingredient for the construction of a Viehweg-Zuo sheaf. In particular, given a line bundle \(\L\) on \(X\) and a cyclic cover \(\phi:Z\to X\) induced by a section of \(\L^q\), there exists a map from a graded complex associated to the filtered D-module \(\w_X\) twisted by some power of \(\L\), to the pushforward of a graded complex associated to \(\w_Z\). We provide a better understanding of this map at the level of Hodge modules.

In order to involve the intersection complex in the construction of a Viehweg-Zuo sheaf, we refine the map described above in terms of the functors in Section \ref{sec: derived pushforward and associated graded}, especially when the cyclic cover \(\phi\) is non-characteristic for a mixed Hodge module. See \cite{Saito88}*{3.5.1} and \cite{Saito90}*{Lemma 2.25} for the details on the pullback of a mixed Hodge module with respect to a non-characteristic morphism.

\begin{thm}
\label{thm: noncharacteristic cycle cover}
Let \(f:X\to Y\) be a morphism of smooth varieties and \(M\in MHM(X)\) be a mixed Hodge module on \(X\). Let \(\phi:Z\to X\) be a proper desingularization of the \(q\)-cyclic cover associated to a section \(\O_X\to \L^q\). Assume that \(\phi\) is non-characteristic for \(M\). Then we have the following morphism in \(D^bG(f^*\gr^F\D_Y)\):
\[
\L^{-j}\tensor_{\O_X}\gr^F\DR_{X/Y}(M)\to \gr^F\DR_{X/Y}(\phi_+\phi^*M), \quad \forall j\ge0.
\]
\end{thm}

\begin{proof}
From Proposition \ref{prop: graded pushforward commuting diagram}, we have \(\gr^F\DR_{X/Y}(\phi_+\phi^*M)=\phi_*\gr^F\DR_{Z/Y}(\phi^*M)\).
By adjunction, it suffices to construct a map
\[
\Phi: \phi^*(\L^{-j}\tensor_{\O_X}\gr^F\DR_{X/Y}(M))\to \gr^F\DR_{Z/Y}(\phi^*M).
\]
Let \(N:=M\tensor \omega_X^{-1}\) be the filtered left D-module associated to \(M\). Since \(\phi\) is non-characteristic for \(M\), \(\phi^*M\) is a mixed Hodge module on \(Z\) whose underlying filtered left D-module is \(\phi^*(N,F)=(\phi^*N,\phi^*F)\). By \eqref{eqn: relative Spencer resolution}, we have
\begin{gather*}
\gr^F\DR_{X/Y}(M)=[\cdots\to (N\tensor\Omega_X^i)\tensor f^*\gr^F\D_Y\to\cdots],\\
\gr^F\DR_{Z/Y}(\phi^*M)=[\cdots\to (\phi^*N\tensor\Omega_Z^i)\tensor (f\circ \phi)^*\gr^F\D_Y\to\cdots].
\end{gather*}
Let \(g:\phi^*\L^{-1}\to \O_Z\) be the induced section. Under a local trivialization of \(\L\), we consider \(g\) as a function, locally on \(Z\). Then, there exists a natural morphism of complexes
\[\Phi_i: \phi^*(\L^{-j}\tensor(N\tensor\Omega_X^i)\tensor f^*\gr^F\D_Y)\to(\phi^*N\tensor\Omega_Z^i)\tensor (f\circ \phi)^*\gr^F\D_Y\]
defined locally as
\[
\Phi_i(\omega\tensor P)=g^j\phi^*(\omega)\tensor P,
\]
where \(\omega\) (resp. \(P\)) is a local section of \(N\tensor \Omega_X^i\) (resp. \((f\circ \phi)^*gr^F\D_Y\)). More precisely, we have the commutativity:
\begin{equation*}
\xymatrix{
{\w\tensor P} \ar@{|->}[r]^{d} \ar@{|->}[d]_{\Phi_i}
& {d(\w\tensor P)} \ar@{|->}[d]_{\Phi_{i+1}}\\
{g^j\phi^*(\w)\tensor P} \ar@{|->}[r]^d
& {d(g^j\phi^*(\w)\tensor P)}
}
\end{equation*}
where \(d\) is the boundary morphism of the complex. Indeed, we have
\[
d(g^j\phi^*(\w)\tensor P)=g^jd(\phi^*(\w)\tensor P),
\]
since \(d\) is \(\O\)-linear after taking the associated graded functor \(\gr^F\). Consequently, the equality
\[
\Phi_{i+1}\left(d(\w\tensor P)\right)=g^jd(\phi^*(\w)\tensor P)
\]
is immediate from the description of the boundary morphisms in Remark \ref{rmk: boundary morphism}. Therefore, \(\Phi_i\) induces a well-defined morphism \(\Phi\) of complexes, which completes the proof.
\end{proof}

For a smooth variety \(X\) of dimension \(n\), the shifted constant sheaf \(\Q_X[n]\) underlies the polarizable Hodge module \(\Q_X^{H}[n]\) of weight \(n\), whose filtered right \(\D_X\)-module is the canonical bundle \(\omega_X\) with the filtration \[F_{\ge-n}(\omega_X)=\omega_X,\quad F_{<-n}(\omega_X)=0.\] We use the same notation \(\Q_X^{H}[n]\) for the underlying filtered \(\D_X\)-module. As an immediate consequence, we obtain the following extension of Popa-Schnell \cite{PS17}*{Proposition 2.8} using the functors in Section \ref{sec: derived pushforward and associated graded}; it is this form that we will need later.

\begin{cor}
\label{cor: cyclic cover and Hodge module}
Let \(f:X\to Y\) be a morphism of smooth varieties and \(\mathcal{L}\) be a line bundle on \(X\). Let \(\phi:Z\to X\) be a proper desingularization of \(q\)-cyclic cover associated to a section \(\O_X\to \L^q\). Then, we have the following morphism in \(D^bG(f^*\gr^F\D_Y)\):
\[
\L^{-j}\tensor_{\O_X}\gr^F\DR_{X/Y}(\Q^{H}_X[n])\to \gr^F\DR_{X/Y}(\phi_+\Q^{H}_Z[n]), \quad \forall j\ge0.
\]
\end{cor}

This follows from the fact that the Hodge module \(\Q_X^{H}[n]\) is non-characteristic for any morphism of smooth varieties.

\subsection{Intersection complexes as Hodge modules on singular varieties}
\label{sec: intersection complex}

On a singular variety \(X\), we highlight the properties of our main ingredient \(\IC_X^H\), the polarizable Hodge module associated to the intersection complex of \(X\). The category of polarizable Hodge modules \(MH(X)\) is defined as follows: given a closed embedding \(X\subset W\) into a smooth variety \(W\), \(MH(X)\) is a subcategory of \(MH(W)\) consisting of polarizable Hodge modules on \(W\) supported on \(X\). It can be shown that the category \(MH(X)\) is defined independent of the choice of an embedding. For example, the minimal perverse extension of the trivial Hodge module \(\Q_{X^{sm}}^{H}[n]\) on the smooth locus \(X^{sm}\) of \(X\) induces the polarizable Hodge module \(\IC_X^{H}\in MH(X)\). See Saito \cite{Saito90} for details.

For a polarizable Hodge module \(M\in MH(X)\), the underlying filtered (right) D-module \(M\) depends on the smooth embedding \(X\subset W\). If we define
\[
p(M):=\min\left\{k\in \Z\mid F_kM\neq0\right\}, \quad S(M):=F_{p(M)}M,
\]
then \(S(M)\) is a coherent \(\O_X\)-module, namely \textit{the lowest nonzero term of the Hodge filtration} of \(M\). Somewhat surprisingly, \(S(M)\) is independent of the choice of an embedding \(X\subset W\) by the first assertion of \cite{Saito90}*{Proposition 2.33}. We aim to understand \(S(\IC_X^H)\), especially when \(X\) has rational singularities.

For that purpose, we highlight Saito's Decomposition Theorem for proper pushforwards of Hodge modules which generalizes Beilinson-Bernstein-Deligne-Gabber's Decomposition Theorem \cite{BBD}.

\begin{thm}[Saito's Decomposition Theorem \cite{Saito91}*{Theorem 2.4}]
\label{thm: Saito's decomposition theorem}
Let \(f:X\to Y\) be a proper morphism of varieties and \(M\in MH(X)\). Then \(\mathscr H^if_+M\in MH(Y)\) and we have a (non-canonical) isomorphism
\[
f_+M\isom \bigoplus_i \mathscr H^if_+M[-i]
\]
in the bounded derived category of mixed Hodge modules \(D^bMHM(Y)\) on \(Y\).
\end{thm}

For the description of the bounded derived category of mixed Hodge modules, see \cite{Saito90}*{Section 4}. Alternatively, we may consider the above isomorphism in the bounded derived category of filtered D-modules, \(D^bF(\D_Y)\), when \(Y\) is smooth. Saito further explains the behavior of the lowest nonzero term of the Hodge filtration for the proper pushforward.

\begin{prop}[\cite{Saito91}*{Proposition 2.6}]
\label{prop: lowest Hodge filtration of pushforward}
In the setting of Theorem \ref{thm: Saito's decomposition theorem}, suppose \(X\) is irreducible and \(M\) is an irreducible polarizable Hodge module with strict support on \(X\). Let \(\mathscr H^if_+M=\oplus_{Z\subset X} M^i_Z\) be the decomposition of the Hodge module by strict support. Then
\[
S(M^i_{f(X)})\isom R^if_*S(M)
\]
and for all \(Z\neq f(X)\), we have \(p(M_Z^i)>p(M)\).
\end{prop}

Notice that every polarizable Hodge module is a direct sum of irreducible polarizable Hodge modules strictly supported on irreducible subvarieties \cite{Saito88}*{5.1.5}.  As an immediate consequence, we have a description of the lowest nonzero term of the Hodge filtration of \(\IC_X^H\) from the desingularization of \(X\).

\begin{prop}
\label{prop: lowest Hodge filtration of intersection complex}
Let \(X\) be an irreducible variety and \(\mu:\widetilde X\to X\) be a proper resolution of singularities. Then 
\[
S(\IC_X^{H})\isom \mu_*\omega_{\widetilde X}.
\]
In particular, if \(X\) has rational singularities, then \(S(\IC_X^{H})\isom \omega_X\).
\end{prop}

\begin{proof}
Since \(\IC_X^{H}\) is an irreducible factor of \(\mathscr H^0\mu_+\left(\Q_{\widetilde X}^{H}[n]\right)\) with strict support \(X\), we obtain the result from Proposition \ref{prop: lowest Hodge filtration of pushforward}.
\end{proof}

Finally, we state an analogue of Corollary \ref{cor: cyclic cover and Hodge module} for the cyclic cover of an irreducible singular variety \(X\) with a line bundle \(\L\). Let \(f:X\to Y\) be a morphism to a smooth variety \(Y\), admitting a smooth embedding \(X\subset W\) over \(Y\):
\begin{equation*}
\xymatrix{
{X} \ar@{^{(}->}[r] \ar[rd]_f
& {W} \ar[d]^p\\
{} & {Y}
}
\end{equation*}
Then, the functor \(\gr^F\DR_{X/Y}\) is defined as:
\[
\gr^F\DR_{X/Y}:=\gr^F\DR_{W/Y}:D^bF(\D_X)\to D^bG(p^*\gr^F\D_Y)
\]
where \(D^bF(\D_X)\) is the bounded derived category of filtered (right) \(\D_W\)-modules supported on \(X\). In particular, for a (mixed) Hodge module \(M\in MHM(X)\), we have
\[
\gr^F\DR_{X/Y}(M)\in D^bG(f^*\gr^F\D_Y)\hookrightarrow D^bG(p^*\gr^F\D_Y),
\]
due to \cite{Saito88}*{Lemme 3.2.6}, and thus,
\[
\gr^F\DR_{X/Y}:D^bMHM(X)\to D^bG(f^*\gr^F\D_Y).
\]
Then we have the following, which plays a crucial role in the construction of a logarithmic Higgs sheaf:

\begin{cor}
\label{cor: cyclic cover and intersection complex}
Under the above notation, let \(\phi:Z\to X\) be a proper desingularization of \(q\)-cyclic cover associated to a section \(\O_X\to \L^q\). Then, we have the following morphism in \(D^bG(f^*\gr^F\D_Y)\):
\[
\L^{-j}\tensor_{\O_X}\gr^F\DR_{X/Y}(\IC^{H}_X)\to \gr^F\DR_{X/Y}(\phi_+\Q^{H}_Z[n]), \quad \forall j\ge0.
\]
\end{cor}

\begin{proof}
Consider the commutative diagram
\begin{equation*}
\xymatrix{
{\widetilde Z} \ar[r]^{\tilde\phi} \ar[d]_{\mu_Z} &{\widetilde X} \ar[d]^{\mu}\\
{Z}\ar[r]^\phi & {X}
}
\end{equation*}
where \(\mu\) and \(\mu_Z\) are resolutions of singularities. Then, \(\tilde \phi\) is a proper desingularization of the cyclic cover associated to \(\O_{\widetilde X}\to \mu^*\L^q\). Consequently, we have
\[
\mu^*\L^{-j}\tensor_{\O_{\widetilde X}}\gr^F\DR_{{\widetilde X}/Y}(\Q^H_{\widetilde X}[n])\to \gr^F\DR_{\widetilde X/Y}(\tilde \phi_+\Q^{H}_{\widetilde Z}[n])
\]
in \(D^b(\mu^*f^*\gr^F\D_Y)\), by Corollary \ref{cor: cyclic cover and Hodge module}. Notice that by Saito's Decomposition Theorem \ref{thm: Saito's decomposition theorem}, we have \(\IC^H_X\) as a direct summand of \(\mu_+\Q^{H}_{\widetilde X}[n]\) and \(\Q^{H}_{Z}[n]\) as a direct summand of \({\mu_Z}_+\Q^{H}_{\widetilde Z}[n]\). Therefore, from Proposition \ref{prop: graded pushforward commuting diagram}, we obtain the conclusion after taking the pushforward \({\mu}_*\).
\end{proof}

\begin{rmk}
\label{rmk: lowest filtration of intersection complex}
From the construction of the functor \(\gr^F\DR_{X/Y}\), we have 
\[
S(M)=\gr^F_{p(M)}\DR_{X/Y}(M).
\]
In particular, we have 
\[
\gr^F_{p(\IC_X^H)}\DR_{X/Y}(\IC_X^H)\isom \mu_*\w_{\widetilde X}
\]
by Proposition \ref{prop: lowest Hodge filtration of intersection complex}.
\end{rmk}

\section{Intersection complexes and Whitney equisingular morphisms}
\label{sec: intersection complex and Whitney equisingular morphism}

As mentioned in the introduction, our ultimate goal is to incorporate the intersection complex in the construction of a Viehweg-Zuo sheaf obtained from a logarithmic Higgs sheaf. In fact, the crucial object appearing in the construction is the graded complex associated to the intersection complex,
\[
\gr^F\DR_{X/Y}(\IC_X^H)\in D^bG(f^*\gr^F\D_Y),
\]
in Corollary \ref{cor: cyclic cover and intersection complex}. It is important to first understand the characteristic variety of \(\IC_X\) via Whitney stratification and determine the support in \(X\times_Y T^*_Y\) of the complex above. In the subsequent section, we study some properties of Whitney equisingular morphisms used later in the construction.

\subsection{Whitney stratifications and the characteristic variety of \texorpdfstring{\(\IC_X\)}{ICX}}
\label{sec: characteristic variety of intersection complex}

For a singular variety \(X\), the intersection complex \(\IC_X\) has two interpretations - one as a regular holonomic D-module and the other as a perverse sheaf. The connection is made via the Riemann-Hilbert correspondence
\[
\DR_X:D^b_{rh}(\D_X)\xrightarrow{\sim}D^b_c(X)
\]
where \(D^b_{rh}(\D_X)\) is the bounded derived category of regular holonomic D-modules and \(D^b_c(X)\) is the bounded derived category of \(\mathbb C\)-constructible sheaves. Here, the category of D-modules on \(X\) is equivalent to the category of \(\D_W\)-modules supported on \(X\) for a smooth embedding \(X\subset W\). In particular, we have:

\begin{thm}[\cite{KS90}*{Theorem 11.3.3.}]
\label{thm: singular support and characteristic variety}
Let \(X\) be a complex manifold and \(M\) be a regular holonomic \(\D_X\)-module. Then 
\[
SS(\DR_X(M))=Ch(M)
\]
where \(SS(\DR_X(M))\) is the singular support of the perverse sheaf \(\DR_X(M)\) and \(Ch(M)\) is the characteristic variety of \(M\).
\end{thm}

We thereby explain a general theory of the singular support of an object in \(D^b_c(X)\) in terms of a Whitney stratification \cite{Whitney65} of the analytic variety \(X\). Then we describe the characteristic variety of \(\IC_X\).

\begin{defn}[\cite{HTT}*{Definition E.3.7.}]
For a complex analytic space (complex algebraic variety) \(X\) embedded in a complex manifold \(W\), a stratification \(X=\coprod_{\alpha\in S}X_\alpha\) by locally closed smooth analytic (algebraic) subsets \(X_\alpha\) is called a \textit{Whitney Stratification} if it satisfies the following Whitney conditions (a) and (b):
\begin{enumerate}[(a)]
    \item For every sequence \(p_i\in X_\alpha\) of points converging to a point \(q\in X_\beta\) (\(\alpha\neq\beta\)) such that the limit \(T\) of the tangent spaces \(T_{p_i}X_\alpha\) exists, we have \(T_qX_\beta\subset T\).
    \item For all sequences of points \(p_i\in X_\alpha\) and \(q_i\in X_\beta\) (\(\alpha\neq\beta\)) with limit \(q\in X_\beta\) such that, in a local chart of \(y\), the limit \(T\) of \(T_{p_i}X_\alpha\) and the limit \(l\) of \(\overline{p_iq_i}\) exist, we have \(l\subset T\).
\end{enumerate}
\end{defn}

It is easy to see that the Whitney conditions (a) and (b) are independent of the choice of an embedding in a complex manifold. In addition, given an embedding into a complex manifold \(X\subset W\) and a Whitney stratification \(X=\coprod_{\alpha\in S}X_\alpha\), we have an induced Whitney stratification \(W=(W\setminus X)\cup\coprod_{\alpha\in S}X_\alpha\).

\begin{thm}[\cite{KS90}*{Proposition 8.4.1.}]
\label{thm: Kashiwara Shapira singular support}
Let \(X=\coprod_{\alpha\in S}X_\alpha\) be a Whitney stratification embedded in a complex manifold \(W\) and let \(\F^\cdot \in D_c^b(X)\) be an object in the bounded derived category of constructible sheaves on \(X\). Then the following conditions are equivalent:
\begin{itemize}
    \item For all \(j\in \Z\), all \(\alpha\in S\), the sheaves \(\mathscr H^j(F^\cdot)|_{X_\alpha}\) are locally constant.
    \item \(SS(F^\cdot)\subset \coprod_{\alpha\in S}T_{X_\alpha}^*W\).
\end{itemize}
Here, \(SS(F^\cdot)\) is the singular support (i.e.  micro-support) of \(F^\cdot\) and \(T_{X_\alpha}^*W\) is a conormal bundle of \(X_\alpha\) in \(W\).
\end{thm}

\begin{rmk}
Kashiwara and Schapira \cite{KS90} originally stated the theorem for \(\mu\)-stratification \(X=\coprod_{\alpha\in S}X_\alpha\) over real manifolds, which is stronger than the Whitney conditions (a) and (b). However, for complex analytic stratifications, the Whitney condition (b) is equivalent to the Kuo-Verdier condition (w) by Teissier \cite{Teissier82} and Henry-Merle \cite{HM83}. The equivalence of the Kuo-Verdier condition (w) and the microlocal condition (\(\mu\)) is proven by Trotman \cite{Trotman89}, so the theorem holds for Whitney stratifications.
\end{rmk}

We note that all of the above hold in the context of complex algebraic varieties. In the following corollary, we state the classical result on the characteristic variety of the intersection complex, whose proof is included for completeness.

\begin{cor}
\label{cor: characteristic variety of intersection complex}
For a complex algebraic variety \(X\) with a Whitney stratification \(X=\coprod_{\alpha\in S}X_\alpha\) embedded in a smooth variety \(W\), we have
\[
Ch(\IC_X)\subset\coprod_{\alpha\in S}T^*_{X_\alpha}W. 
\]
\end{cor}

\begin{proof}
From the definition of the intersection complex via a minimal extension (i.e. Deligne-Goresky-MacPherson extension), \(\mathscr H^j(\IC_X)|_{X_\alpha}\) is locally constant for all \(\alpha\in S\), \(j\in \Z\). Therefore, by Theorem \ref{thm: singular support and characteristic variety} and \ref{thm: Kashiwara Shapira singular support}, we obtain the result.
\end{proof}

At last, combining the results above, we obtain the following proposition on the support of the graded complex associated to \(\IC_X^H\). Notice that the support of an object in \(D^bG(f^*\gr^F\D_Y)\) is a subvariety in \(X\times_YT^*_Y\), since it is the relative Spec of the sheaf of \(\O_X\)-algebra \(f^*\gr^F\D_Y\) on \(X\).

\begin{prop}
\label{prop: support of graded complex of IC}
Let \(f:X\to Y\) be a projective Whitney equisingular morphism. Then, we have
\[
\supp \left(\gr^F\DR_{X/Y}(\IC_X^H)\right)=X\times_YY\subset X\times_YT^*_Y.
\]
In other words, the support is the zero section of the vector bundle \(f^*T^*_Y\) on \(X\).
\end{prop}

\begin{proof}
Assume that \(f\) factors through a smooth embedding \(X\subset W\), where \(W\) is smooth over \(Y\). Since \(f\) is Whitney equisingular, there exists a Whitney stratification
\[
X=\coprod_{\alpha\in S}X_\alpha
\]
such that \(f|_{X_\alpha}\) is smooth for all \(\alpha\in S\). Notice that the filtered Spencer resolution of \(\IC_X^H\) as a filtered \(\D_W\)-module in \(W\) is supported in \(Ch(\IC_X)\subset T^*_W\). From the construction of the functor \(\gr^F\DR_{X/Y}\) in Section \ref{sec: derived pushforward and associated graded} and \ref{sec: intersection complex}, we have
\[
\supp \left(\gr^F\DR_{X/Y}(\IC_X^H)\right)\subset \left(W\times_Y T^*_Y\right)\cap Ch(\IC_X),
\]
because the functor \(\gr^F\DR_{X/Y}(\IC_X^H)\) pulls back the Spencer resolution to \(W\times_Y T^*_Y\). From the fact that \(X_\alpha\) is smooth over \(Y\), we can easily check that
\[
\left(W\times_Y T^*_Y\right)\cap T_{X_\alpha}^*W=X_\alpha\times_YY\subset W\times_Y T^*_Y.
\]
Therefore, Corollary \ref{cor: characteristic variety of intersection complex} implies the inclusion
\[
\supp \left(\gr^F\DR_{X/Y}(\IC_X^H)\right)\subset X\times_YY,
\] and the equality follows from Proposition \ref{prop: lowest Hodge filtration of intersection complex} and Remark \ref{rmk: lowest filtration of intersection complex}.
\end{proof}

\subsection{Properties of Whitney equisingular morphisms}
\label{sec: Whitney equisingular morphism}

In Definition \ref{defn: Whitney equisingularity}, Whitney equisingularity seems to depend on the choice of a Whitney stratification on \(X\). However, Teissier proves the existence of the \textit{coarsest} Whitney stratification, which allows us to check Whitney equisingularity for this stratification. For a relatively modern treatment of the theorem, see Lipman \cite{Lipman00}*{Proposition 2.7}.

\begin{thm}[Teissier \cites{Teissier82}]
\label{thm: coarsest Whitney stratification}
There exists the coarsest (i.e. minimal) Whitney stratification \((X_\alpha)_{\alpha\in S}\) for a complex analytic space (or a complex algebraic variety) \(X\). In other words, every Whitney stratification of \(X\) is a refinement of \((X_\alpha)_{\alpha\in S}\).

Moreover, for a complex algebraic variety, the coarsest analytic Whitney stratification is algebraic.
\end{thm}

\begin{cor}
\label{cor: coarsest Whitney equisingular}
A morphism \(f:X\to Y\) is Whitney equisingular if and only if \(f|_{X_\alpha}:X_\alpha\to Y\) is smooth for all \(\alpha\in S\), where \(X=\coprod_{\alpha\in S}X_\alpha\) is the coarsest Whitney stratification of \(X\). 
\end{cor}

\begin{proof}
Suppose there exists a Whitney stratification \(X=\coprod_{\alpha'\in S'}X_{\alpha'}\) such that \(X_{\alpha'}\) is smooth over \(Y\). Let \(x\in X_\alpha\). Then there exists \(\alpha'\in S\) such that \(x\in X_{\alpha'}\subset X_\alpha\). Since \(f|_{X_{\alpha'}}:X_{\alpha'}\to Y\) is a submersion at \(x\), \(f|_{X_\alpha}:X_\alpha\to Y\) is a submersion at \(x\) as well. Therefore, \(f|_{X_\alpha}\) is smooth. 
\end{proof}

It is important to notice that every morphism is generically Whitney equisingular. Indeed, for an arbitrary morphism \(f:X\to Y\), there exists an open subvariety \(V\subset Y\) such that \(f|_{f^{-1}(V)}\) is Whitney equisingular, by the generic smoothness of \(f|_{X_\alpha}\) for all \(\alpha\). Additionally, as an immediate consequence of Corollary \ref{cor: coarsest Whitney equisingular}, Whitney equisingularity is a local property: if \(X\to Y\) is locally Whitney equisingular, i.e. Whitney equisingular on a neighborhood of \(x\in X\) and \(y\in Y\), then it is globally Whitney equisingular.

Furthermore, the following shows that Whitney equisingularity is preserved under taking a fiber product.

\begin{prop}
\label{prop: Whitney equisingularity: fiber product}
Let \(f:X_1\to Y\) and \(f_2:X_2\to Y\) be a Whitney equisingular morphism. Then \(f_1\times_Y f_2:X_1\times_Y X_2\to Y\) is Whitney equisingular.
\end{prop}

Before proving this statement, we state the stability of the Whitney conditions (a) and (b) under a transversal pullback.

\begin{prop}[\cite{Schurmann03}*{Section 4.3.}]
\label{prop: transversal pullback}
Let \(f: W\to V\) be a holomorphic mapping of complex manifolds. Suppose \(f\) is transversal (i.e. non-characteristic) to a Whitney stratification \(V=\coprod_{\alpha\in S}V_{\alpha}\): for every point \(w\in W\),
\[
df(w)(T_wW)+T_{f(w)}V_\beta=T_{f(w)}V
\]
where \(V_{\beta}\) is a stratum containing \(f(w)\). Then, \(W=\coprod_{\alpha\in S}f^{-1}(V_\alpha)\) is a Whitney stratification.
\end{prop}

\begin{proof}[Proof of Proposition \ref{prop: Whitney equisingularity: fiber product}]
Suppose \(i_j:X_j\hookrightarrow W_j\) is a closed embedding into a smooth variety and \(g_j:W_j\to Y\) is a smooth morphism such that \(f_j=g_j\circ i_j\) (\(j=1,2\)). Since Whitney equisingularity is a local property, we only need to prove the statement in such cases. Then, we have a Whitney stratification
\[
W_j=(W_j\setminus X_j)\cup \coprod_{\alpha_j\in S_j}X_{j,\alpha_j}
\]
such that \(X_{j,\alpha_j}\) is smooth over \(Y\) for all \(\alpha_j\in S_j\). Consider a closed embedding \(\rho:W_1\times_Y W_2\hookrightarrow W_1\times W_2\). It is easy to see that \(\rho\) is transversal to the Whitney stratification of \(W_1\times W_2\) induced by those of \(W_1\) and \(W_2\), described above. Therefore, by Proposition \ref{prop: transversal pullback}, 
\[
X_1\times_Y X_2=\coprod_{\alpha_1\in S_1, \alpha_2\in S_2}X_{1,\alpha_1}\times_Y X_{2,\alpha_2}
\]
is a Whitney stratification. Since \(X_{1,\alpha_1}\times_Y X_{2,\alpha_2}\) is smooth over \(Y\), we are done.
\end{proof}

\begin{rmk}
Using the same strategy, we can easily show that Whitney equisingularity is stable under a pullback; for a morphism \(g:Y'\to Y\) of smooth varieties, the induced morphism \(X\times_Y Y'\to Y'\) is Whitney equisingular if \(f:X\to Y\) is Whitney equisingular.
\end{rmk}

\section{Construction of Viehweg-Zuo sheaves}
\label{sec: construction of Viehweg-Zuo sheaf}

For a smooth quasi-projective variety \(V\) with a smooth projective compactification \(Y\) and boundary a simple normal crossing divisor \(D\), a Viehweg-Zuo sheaf is a big line bundle (or more generally, a coherent sheaf which is big in the sense of Viehweg)
\[
A\subset \Omega_Y(\log D)^{\tensor N}
\]
for some positive integer \(N\). The existence of a Viehweg-Zuo sheaf implies that \(V\) is of log general type by Campana-P\u aun \cite{CP19}*{Theorem 7.11}. We explain how a Viehweg-Zuo sheaf is constructed from a logarithmic Higgs sheaf on \(Y\) with poles along \(D\) endowed with some positivity properties.

A \textit{graded Higgs sheaf} \(\F_\bullet=\oplus_k\F_k\) on \(Y\) is a graded \(\O_Y\)-module with a Higgs structure
\[
\phi_\bullet:\F_{\bullet}\to \F_{\bullet+1}\tensor \Omega_Y
\]
such that \(\phi\wedge \phi:\F_{\bullet}\to \F_{\bullet+2}\tensor \Omega_Y^2\) is the zero morphism. Equivalently, \(\F_\bullet\) is a graded \(\gr^F\D_Y\)-module. The \textit{support} of a coherent \(\gr^F\D_Y\)-module \(\F_\bullet\) is a Zariski closed subset of the cotangent bundle \(T_Y^*\). Furthermore, a \textit{graded logarithmic Higgs sheaf} with poles along \(D\) is a graded \(\O_Y\)-module \(\F_\bullet=\oplus_k\F_k\) with a logarithmic Higgs structure
\[
\phi_\bullet:\F_{\bullet}\to \F_{\bullet+1}\tensor \Omega_Y(\log D)
\]
such that \(\phi\wedge \phi=0\).

Now, we explain Popa and Schnell's construction of graded logarithmic Higgs sheaves via Hodge modules. To begin with, let \((\mathcal V,F)\) be a polarizable variation of Hodge structure on an open subvariety of \(Y\) such that \(\mathcal V\) is a flat vector bundle and \(F\) is the Hodge filtration. Shrinking this open subvariety, we assume that \((\mathcal V,F)\) is defined on the complement of a divisor \(D'\supset D\) in \(Y\).

Notice that \((\mathcal V,F)\) on \(Y\setminus D'\) admits Deligne's canonical extension \((\mathcal V_0,F)\) over \(Y^o\subset Y\), the complement of a codimension \(2\) subvariety, whose eigenvalues of residues are contained in \([0,1)\). Indeed, \(Y^o\) can be the complement of the singular locus of \(D'\) in \(Y\) so that \(D'|_{Y^o}\) is a smooth divisor. While the associated graded module \(\gr^F_\bullet\mathcal V_0\) is a graded logarithmic Higgs sheaf on \(Y^o\) with poles along \(D'\), Popa and Schnell provide a method to build that on \(Y\) with poles along \(D\).

\begin{thm}[\cite{PS17}*{Section 2.8}]
\label{thm: construction of logarithmic Higgs sheaf}
Under the above notation, let \(M\) be a polarizable Hodge module on \(Y\) associated to \((\mathcal V,F)\) and \((N,F)\) be the underlying filtered left \(\D_Y\)-module. Assume \(D\subset D'\) where \(D\) is an snc divisor. Suppose there is a sub-graded Higgs sheaf \(\G_\bullet\subset \gr^F_\bullet N|_{Y^o}\) on \(Y^o\subset Y\), the complement of a codimension \(2\) subvariety. If \(\G_\bullet|_{Y^o\setminus D}\) is supported on the zero section of the cotangent bundle \(T^*_{Y^o\setminus D}\), then
\[
\F_\bullet:=(\G_\bullet\cap \gr^F_\bullet \mathcal V_0)^{**}
\]
is a graded logarithmic Higgs sheaf on \(Y\) with poles along \(D\).
\end{thm}

Here, \((\cdot)^{**}\) denotes the reflexive hull of the sheaf. The main idea of the proof is to analyze the Kashiwara-Malgrange V-filtration along \(D'\). See \cite{PS17} for complete details. Furthermore, Popa and Schnell explain that when the lowest nonzero graded piece \(\G_p\) of \(\G_\bullet\) contains a sufficiently positive line bundle, we have a Viehweg-Zuo sheaf.

\begin{thm}
\label{thm: Viehweg-Zuo sheaf from Higgs sheaf}
Under the assumptions of Theorem \ref{thm: construction of logarithmic Higgs sheaf}, let \(\L\subset\G_p\) be a line bundle on \(Y^o\) contained in the lowest nonzero graded piece \(\G_p\) of \(\G_\bullet\). If \(\L(-D)\) is big, then there exists a big line bundle \(A\) with
\[
A\subset \Omega_Y(\log D)^{\tensor N}
\]
for some positive integer \(N\). In particular, \(\omega_Y(D)\) is big.
\end{thm}

\begin{proof}[Sketch of the proof]
The proof of \cite{PS17}*{Proposition 2.15} shows that \(\L(-D)\subset \F_p\). We consider the following chain of coherent \(\O_Y\)-module homomorphisms:
\[
0\to \F_p\xrightarrow{\phi}\F_{p+1}\tensor \Omega_Y(\log D)\xrightarrow{\phi\tensor id}\F_{p+2}\tensor \Omega_Y(\log D)^{\tensor2} \to \cdots
\]
The dual of the kernel of \(\phi_\bullet\) is weakly positive by Zuo \cite{Zuo00},  Brunebarbe \cite{Brunebarbe18}, and Popa-Wu \cite{PW16}. Therefore, we eventually have
\[
\L(-D)\subset \ker \phi_{p+k}\tensor \Omega_Y(\log D)^{\tensor k}
\]
for some positive integer \(k\), from which we obtain a big line bundle
\[
A\subset \Omega_Y(\log D)^{\tensor N}.
\]
Thus, \(\w_Y(D)\) is big by \cite{CP19}*{Theorem 7.11}.
\end{proof}

\section{Proof of Theorem \ref{thm: Viehweg hyperbolicity for Gorenstein rational singularities}}
\label{sec: proof of the theorem}

Let \(Y\) be a good compactification of \(V\) with boundary a simple normal crossing divisor \(D\). We aim to prove that \(\omega_{Y}(D)\) is a big line bundle via the construction of a Viehweg-Zuo sheaf, as in the previous section. The following lemma is the analogue of Popa and Schnell's version of Viehweg's fiber product trick in \cite{PS17}*{Section 4.2}, for a flat family \(f:U\to V\) of projective varieties with Gorenstein rational singularities.

\begin{lem}
\label{lem: main construction}
For any line bundle \(A\) on \(Y\), there exists a projective variety \(X\) with \(\Q\)-Gorenstein canonical singularities and a morphism \(g:X\to Y\) with the following properties.
\begin{enumerate}
    \item \(g|_{g^{-1}(V)}:g^{-1}(V)\to V\) is a flat Whitney equisingular family of projective varieties with Gorenstein rational singularities.
    \item On the complement of a codimension 2 subvariety of \(Y\), there exists a positive integer \(q\) such that we have an inclusion
    \[
    A^{\tensor q}\hookrightarrow {\tilde g}_*\omega_{\widetilde X/Y}^{\tensor q},
    \]
    where \(\mu:\widetilde X\to X\) is a desingularization and \(\tilde g=g\circ \mu\).
\end{enumerate}
\end{lem}

\begin{proof}
To begin with, there exists a projective compactification \(X\) of \(U\) such that \(f:U\to V\) extends to a morphism \(g:X\to Y\), so that \(g|_{U}=f\). We modify \(X\) and \(g\) to satisfy the properties (1) and (2).

Let \(\mu:\widetilde X\to X\) be a desingularization and \(\tilde g:=g\circ \mu:\widetilde X\to Y\). Recall the results from \cite{PS17}*{Section 4.2}, which is essentially an application of semistable reduction and Viehweg's fiber product trick (e.g. Mori \cite{Mori87}*{Section 4}): there exists positive integers \(q\) and \(r\) such that for any resolution
\[
\tilde g^{(r)}:\widetilde X^{(r)}\to Y
\]
of \(r\)-fold fiber product \(\tilde g^r:\widetilde X^r\to Y\) of \(\tilde g\), we have an injection
\begin{equation}
\label{eqn: inclusion to pushforward pluricanonical}
A^{\tensor q}\hookrightarrow {\tilde g^{(r)}}_*\omega_{\widetilde X^{(r)}/Y}^{\tensor q}.
\end{equation}
on the complement of a codimension 2 subvariety of \(Y\). Here, \(r\)-fold fiber product \(X^r\) refers to the main irreducible component of
\[
X\times_YX\times_Y \dots \times_YX
\]
iterated \(r\)-times. We denote by \(\mu^r:\widetilde X^{(r)}\to X^r\), a resolution of singularities, which factors through \(\widetilde X^r\).

Let \(X'\) be the relative canonical model of \(X^r\), i.e. the relative Proj of the sheaf of relative canonical rings with respect to a desingularization (see Birkar-Cascini-Hacon-McKernan \cite{BCHM}). Let \(g':X'\to Y\) be the induced morphism. Recall that a product of varieties with Gorenstein rational singularities has Gorenstein rational singularities. Consequently, \(U^r\) already has Gorenstein canonical singularities, so that
\[
g'^{-1}(V)=U^r(:=U\times_VU\times_V\dots\times_VU)\subset X^r.
\]
Thus,
\[
g'|_{g'^{-1}(V)}:g'^{-1}(V)\to V
\]
is a flat Whitney equisingular family of projective varieties with Gorenstein rational singularities, by Proposition \ref{prop: Whitney equisingularity: fiber product}.

Therefore, we replace \(X\) and \(g\) with \(X'\) and \(g'\); the property (1) is immediate. Additionally, from the construction, \(X'\) has \(\Q\)-Gorenstein canonical singularities and \(\mu^r\) factors through \(X'\). Hence, \eqref{eqn: inclusion to pushforward pluricanonical} implies the property (2).
\end{proof}

By adjunction, the condition (2) is equivalent to saying that
\[
(\omega_{\widetilde X/Y}\tensor \tilde g^*A^{-1})^{\tensor q}
\]
has a nonzero section over the complement of a codimension 2 subvariety of \(Y\). This continues to hold for any multiple of \(q\).

\begin{proof}[Proof of Theorem \ref{thm: Viehweg hyperbolicity for Gorenstein rational singularities}]
We choose a line bundle \(A\) on \(Y\) such that \(A(-D)\) is big. Let \(g:X\to Y\) be a morphism obtained from Lemma \ref{lem: main construction}, satisfying the properties (1) and (2). We redefine \(U:=g^{-1}(V)\) and \(\widetilde U:=\tilde g^{-1}(V)\). By taking a multiple of \(q\), we assume \(\omega_{X}^{\tensor q}\) is a line bundle and
\[
(\omega_{\widetilde X/Y}\tensor \tilde g^*A^{-1})^{\tensor q}
\]
has a nonzero section over \(Y^o\), the complement of a codimension \(2\) subvariety in \(Y\). Since \(X\) has canonical singularities, we have
\[
\omega_{\widetilde X}^{\tensor q}\isom\mu^*\omega_{X}^{\tensor q}(\widetilde E)
\]
where \(\widetilde E\) is an effective exceptional divisor of \(\mu\). Hence, we have a nonzero section of
\[
(\omega_{X/Y}\tensor g^*A^{-1})^{\tensor q}\isom\mu_*\left((\omega_{\widetilde X/Y}\tensor \tilde g^*A^{-1})^{\tensor q}\right)
\]
over \(Y^o\).
Write \(\widetilde E=E+(\widetilde E-E)\) where \(E\) is the sum of divisors mapping nontrivially to \(g^{-1}(V)\). Since \(U=g^{-1}(V)\) has Gorenstein rational singularities, all the coefficients of \(E\) are some multiple of \(q\), and thus \(E=qE_0\), where \(\omega_{\widetilde U}(-E_0)\isom\mu^*\omega_{U}\). As a consequence, we have
\[
(\omega_{X/Y}\tensor g^*A^{-1})^{\tensor q}\isom\mu_*\left((\omega_{\widetilde X/Y}(-E_0)\tensor \tilde g^*A^{-1})^{\tensor q}\right),
\]
and thus,
\(
(\omega_{\widetilde X/Y}(-E_0)\tensor \tilde g^*A^{-1})^{\tensor q}
\)
has a nonzero section \(s\) over \(Y^o\).

Note that a logarithmic Higgs sheaf constructed on \(Y^o\) uniquely extends to that on \(Y\) via taking the reflexive hull as in Theorem \ref{thm: construction of logarithmic Higgs sheaf}. Therefore, without loss of generality, we replace \(X,Y,U,V\) with the respective preimages of \(Y^o\). Denote \(\L:=\omega_{\widetilde X/Y}(-E_0)\tensor \tilde g^*A^{-1}\) and let \(\tilde\phi:Z\to \widetilde X\) be a desingularization of \(q\)-cyclic cover associated to the section \(s:\O_{\widetilde X}\to \L^{\tensor q}\).
\begin{displaymath}
\xymatrix{
{Z}\ar[r]^-{\tilde\phi}\ar[rrd]_{h}\ar@(ur,ul)[rr]^{\phi}& {\widetilde X}\ar[r]^-{\mu} \ar[rd]^{\tilde g} & {X} \ar[d]^{g} \\
& & {Y}
}
\end{displaymath}
By Corollary \ref{cor: cyclic cover and Hodge module}, we have
\[
\L^{-1}\tensor \gr^F\DR_{\widetilde X/Y}(\Q_{\widetilde X}^{H}[n])\to \gr^F\DR_{\widetilde X/Y}(\tilde\phi_+\Q_{Z}^{H}[n])
\]
in \(D^bG(\tilde g^*\gr^F\D_{Y})\). Taking the derived pushforward \(\tilde g_*\), we obtain a morphism
\begin{equation*}
\tilde g_*\left(\L^{-1}\tensor \gr^F\DR_{\widetilde X/Y}(\Q_{\widetilde X}^{H}[n])\right) \to \gr^F\DR_{Y/Y}(h_+\Q_{Z}^{H}[n])\isom \gr^F(h_+\Q_{Z}^{H}[n])
\end{equation*}
in \(D^bG(\gr^F\D_{Y})\), by Proposition \ref{prop: graded pushforward commuting diagram}. Due to Saito's Decomposition Theorem \ref{thm: Saito's decomposition theorem}, \(h_+\Q_{Z}^{H}[n]\) decomposes into a direct sum of Hodge modules. Let the polarizable Hodge module \(M\) be the direct summand of \(\mathscr H^0(h_+\Q_{Z}^{H}[n])\) associated to the variation of Hodge structure induced by the middle cohomology of \(h\). Then, we have the following morphism of graded \(\gr^F\D_Y\)-modules:
\begin{equation}
\label{eqn: Popa Schnell G}
\mathscr H^0\left(\tilde g_*\left(\L^{-1}\tensor \gr^F\DR_{\widetilde X/Y}(\Q_{\widetilde X}^{H}[n])\right)\right) \to \gr^FM.
\end{equation}
Here, \(M\) is considered as the filtered right \(\D_{Y}\)-module associated to the Hodge module. Let \(\G_\bullet\) be its image.

As in \cite{PS17}*{Proposition 2.9}, the lowest nonzero graded piece \(\G_p\) of \(\G_\bullet\) is the image of the morphism
\[
\tilde g_*(\L^{-1}\tensor \omega_{\widetilde X})\to h_*\omega_{Z}.
\]
which is injective, due to the injectivity of \(\tilde\phi^*(\L^{-1}\tensor \omega_{\widetilde X})\to \omega_{Z}\) and the adjunction. Therefore, \(\G_p=\tilde g_*(\L^{-1}\tensor \omega_{\widetilde X})\isom\omega_{Y}\tensor A\).

On a different note, observe that
\[
\L|_{\widetilde U}=\mu^*(\omega_{U/V}\tensor  g^* A^{-1}).
\]
Let \(\L_U:=\omega_{U/V}\tensor  g^* A^{-1}\) so that \(\L|_{\widetilde U}=\mu^*\L_U\). From Corollary \ref{cor: cyclic cover and intersection complex} and its proof, there exists a morphism of graded \(\gr^F\D_V\)-modules,
\begin{equation}
\label{eqn: Popa Schnell IC}
\mathscr H^0\left(g_*\left(\L_U^{-1}\tensor \gr^F\DR_{U/V}(\IC^H_U)\right)\right)\to \gr^FM|_V,
\end{equation}
which is obtained by the same strategy used to obtain \eqref{eqn: Popa Schnell G}. Furthermore, it is immediate that \eqref{eqn: Popa Schnell IC} factors through \eqref{eqn: Popa Schnell G} restricted to \(V\).

From the Whitney equisingularity of \(g:U\to V\) and Proposition \ref{prop: support of graded complex of IC}, the left hand side of \eqref{eqn: Popa Schnell IC} is supported on the zero section of the cotangent bundle \(T_{V}^*\). Consequently, its image \(\H_\bullet\) in \(\gr^F_\bullet M|_{V}\) is supported on the zero section of \(T_{V}^*\), and we have the following inclusions
\[
\H_\bullet\subset \G_\bullet|_{V}\subset \gr^F_\bullet M|_{V}.
\]
By Proposition \ref{prop: lowest Hodge filtration of intersection complex} and Remark \ref{rmk: lowest filtration of intersection complex}, we have
\[
\gr^F_p\DR_{U/V}(\IC_{U}^{H}))\isom \w_{U}
\]
so that the lowest nonzero graded piece \(\H_p\) of \(\H_\bullet\) satisfies
\[
\H_p=\G_p|_{V}\isom\w_{V}\tensor A\subset \gr^F_pM|_{V}.
\]

Let \(j:V\to Y\) be the open embedding. Define a graded (right) \(\gr^F\D_{Y}\)-module
\[
\mathcal E_\bullet:=j_*\H_\bullet \cap \G_\bullet\subset j_*j^*\G_\bullet.
\]
Then, the lowest nonzero graded piece satisfies \(\mathcal E_p\isom\w_{Y}\tensor A\), and additionally \(\mathcal E_\bullet\subset \gr^F_\bullet M\).
Via the left-right correspondence of D-modules, we take a tensor product \(\cdot\tensor \omega_{Y}^{-1}\):
\[
\mathcal E^l_\bullet:=\mathcal E_\bullet\tensor \w_Y^{-1}\subset \gr^F_\bullet N
\]
where \(N\) is the underlying filtered left module of \(M\). Still, \(\mathcal E^l_\bullet|_V\) is supported on the zero section of \(T^*_V\), and thus by Theorem \ref{thm: construction of logarithmic Higgs sheaf}, we obtain a graded logarithmic Higgs sheaf \(\F_\bullet\) on \(Y\) with poles along \(D\).

Furthermore, the lowest nonzero graded piece of \(\mathcal E_\bullet^l\) is \(A\). Recall that we have started with the line bundle \(A\) such that \(A(-D)\) is big. Therefore, Theorem \ref{thm: Viehweg-Zuo sheaf from Higgs sheaf} implies that \(\w_Y(D)\) is big.
\end{proof}

\begin{rmk}
In fact, the same proof works if \(f:U\to V\) satisfies the assumptions in Theorem \ref{thm: Viehweg hyperbolicity for Gorenstein rational singularities} away from a closed subset of codimension at least \(2\) in \(V\), but we do not have a gain in the conclusion. Notice that \(V\) is of log general type if and only if \(V^o\) is of log general type, where \(V^o\subset V\) is the complement of a codimension at least \(2\) subvariety.
\end{rmk}

\section{Application to families with isolated hypersurface singularities}
\label{sec: isolated hypersurface singularity}

This section is devoted to applications of Corollary \ref{cor: Viehweg hyperbolicity for isolated hypersurface singularities}. We first establish the numerical criterion of Whitney equisingularity in terms of the Milnor sequence; then, we use it to analyze certain strata of the moduli space of varieties of general type, and families of surfaces and threefolds with non-negative Kodaira dimension.

\subsection{Numerical criterion for Whitney equisingularity}
\label{subsec: numerical criteria}

We recall the definition of the Milnor sequence given in the introduction: for a germ \((X_0,x_0)\subset (\C^{n+1},0)\) of an isolated hypersurface singularity, the Milnor sequence
\[
\mu^{(*)}_{x_0}(X_0)=\left(\mu^{(n+1)}_{x_0}(X_0),\dots, \mu^{(i)}_{x_0}(X_0),\dots, \mu^{(0)}_{x_0}(X_0)\right)
\]
consists of the Milnor numbers \(\mu^{(i)}_{x_0}(X_0):=\mu_{x_0}(X_0\cap H)\), where \(H\) is a general plane of dimension \(i\) through \(x_0\). For a deformation of isolated hypersurface singularities, the numerical criterion of Whitney equisingularity was established by Teissier, Brian\c con, and Speder, via the constancy of the Milnor sequences \(\mu^{(*)}\).

\begin{thm} [\cites{Teissier73, BS76}]
\label{thm: deformation milnor sequence Whitney}
Let \(f:\mathscr X\to \mathscr Y\) be an analytic deformation of an isolated hypersurface singularity \((\mathscr X_0,x_0)\) with a section \(\sigma:\mathscr Y\to\mathscr X\) such that \(f|_{\mathscr X\setminus \sigma(\mathscr Y)}\) is smooth. Then \(\mathscr X=\mathscr X\setminus \sigma(\mathscr Y)\cup \sigma(\mathscr Y)\) is a Whitney stratification if and only if \(\mu_{\sigma(y)}^{(*)}(\mathscr X_y)\) is constant for all \(y\in \mathscr Y\)
\end{thm}

See Brian\c con-Speder \cite{BS76}*{Corollary 3} for the proof. Using this theorem, we give a criterion for flat families of varieties with isolated hypersurface singularities to be Whitney equisingular.

\begin{thm}
\label{thm: Whitney equisingular, constant Milnor sequence}
Let \(Y\) be a smooth variety and \(f:X\to Y\) be a flat family of projective varieties with isolated hypersurface singularities. Then, \(f\) is Whitney equisingular if and only if each nonzero Milnor sequence \(\mu^{(*)}\) appears a fixed number of times on every fiber.
\end{thm}

In particular, this implies Corollary \ref{cor: Viehweg hyperbolicity for isolated hypersurface singularities}.

\begin{proof}
Let \(C_{f,\mathrm{red}}\) be the reduced critical locus of \(f\). Then, it is easy to see that \(f\) is Whitney equisingular if and only if
\[
X=X\setminus C_{f,\mathrm{red}}\coprod C_{f,\mathrm{red}},
\]
is a Whitney stratification and \(C_{f,\mathrm{red}}\) is \'etale over \(Y\). 

Suppose first that \(f\) is Whitney equisingular. Then, by Theorem \ref{thm: deformation milnor sequence Whitney}, the number of appearances of a particular Milnor sequence in the fiber \(X_y\) is locally analytically constant for \(y\in Y\). As a consequence, we have the conclusion.

Now, it suffices to prove the converse, assuming that each nonzero Milnor sequence appears a fixed number of times on every fiber. In particular, the sum of Milnor numbers
\begin{equation}
\label{eqn: sum of milnor numbers}
\sum_{x_k\in X_y} \mu_{x_k}(X_y)   
\end{equation}
is constant for all \(y\in Y\), where the sum is taken over the singular points \(x_k\) of the fiber \(X_y=\pi^{-1}(y)\). By Teissier's non-splitting principle \cite{Teissier75}*{Theorem 5 p.615}, \(C_{f,\mathrm{red}}\) is \'etale over \(Y\). Moreover, due to the analytic upper semicontinuity of the Milnor sequences \cite{Teissier75}*{2.7} (or \cite{Teissier77}*{5.12.3}) explained below, the Milnor sequence is locally analytically constant along \(C_{f,\mathrm{red}}\). Thus, we can apply Theorem \ref{thm: deformation milnor sequence Whitney} to conclude that \(f\) is Whitney equisingular.
\end{proof}

Teissier's non-splitting principle states that given a germ of an analytic deformation \(f:\mathscr X\to \mathscr Y\) of an isolated hypersurface singularity \((\mathscr X_0,x_0)\), the sum of Milnor numbers in \eqref{eqn: sum of milnor numbers} is analytically upper semicontinuous, and if the sum is constant, then \(C_{f,\mathrm{red}}\) is analytically isomorphic to \(\mathscr Y\). Hence, the \(\mu\)-constant stratum in \(\mathscr Y\) is an analytic locally closed subset of \(\mathscr Y\). Furthermore, the results in \cite{Teissier75}*{Section 2} (c.f. \cite{Teissier77}*{5.12.3}) prove that \(\mu^{(i)}\) is analytically upper semicontinuous for all \(i\) and the \(\mu^{(*)}\)-constant stratum in \(\mathscr Y\) is also well-defined as an analytic locally closed subset of \(\mathscr Y\). This is referred to as the analytic upper semicontinuity of the Milnor sequences.

As a consequence, given a flat family \(f:X\to Y\) of projective varieties with isolated hypersurface singularities and \(Y\) a possibly singular variety, \(Y\) admits a complex analytic stratification where each stratum
\[
Y^{\{\mu^{(*)}_s\}}:=\left\{y\in Y | \{\mu_{x_k}^{(*)}(X_y)\}=\{\mu^{(*)}_s\} \textrm{ as a multiset}\right\}
\]
consists of the points whose fibers have the same multiset \(\{\mu^{(*)}_s\}_{s\in S}\) of nonzero Milnor sequences appearing on the isolated hypersurface singularities. Here, a multiset is a set of elements counted with multiplicities. For instance, the number of singularities is fixed along the fibers over points in \(Y^{\{\mu^{(*)}_s\}}\). We prove that the above stratification is algebraic, i.e. each \(Y^{\{\mu^{(*)}_s\}}\) is a Zariski locally closed subset of \(Y\).

\begin{lem}
Under the above notation, \(Y^{\{\mu^{(*)}_s\}}\) is a Zariski locally closed subset in \(Y\), if non-empty.
\end{lem}

\begin{proof}
By the Noetherian induction, it suffices to prove that the largest stratum, say \(Y^{\{\mu^{(*)}_s\}}\), is a Zariski open subset of \(Y\). By the base change to a desingularization, we may assume \(Y\) smooth. Then, by Theorem \ref{thm: Whitney equisingular, constant Milnor sequence} and its proof, \(Y^{\{\mu^{(*)}_s\}}\) is the set of points \(y\in Y\) such that \(f:X\to Y\) is Whitney equisingular over an analytic neighborhood of \(y\in Y\). By Corollary \ref{cor: coarsest Whitney equisingular}, if \(f:X\to Y\) is Whitney equisingular over an analytic neighborhood of \(y\in Y\), then over a Zariski neighborhood of \(y\in Y\). Hence, \(Y^{\{\mu^{(*)}_s\}}\) is a Zariski open subset of \(Y\).
\end{proof}

As an immediate consequence, given a multiset \(\{\mu^{(*)}_s\}_{s\in S}\) of nonzero Milnor sequences, we have the moduli stack \(\M_{n,v}^{\{\mu^{(*)}_s\}}\) of KSB-stable varieties of dimension \(n\) and volume \(v\), with isolated rational hypersurface singularities whose multiset is \(\{\mu^{(*)}_s\}_{s\in S}\). If non-empty, this is a locally closed substack of \(\M_{n,v}\), the moduli stack of KSB-stable varieties of dimension \(n\) and volume \(v\):
\[
\M_{n,v}^{\{\mu^{(*)}_s\}}\hookrightarrow \M_{n,v}.
\]
Therefore, \(\M_{n,v}^{\{\mu^{(*)}_s\}}\) is hyperbolic in a birational geometric sense: for every generically finite map \(V\to \M_{n,v}^{\{\mu^{(*)}_s\}}\), \(V\) is of log general type.

\subsection{Families of surfaces and threefolds}
\label{subsec: surfaces and threefolds}

In what follows, we apply Corollary \ref{cor: Viehweg hyperbolicity for isolated hypersurface singularities} for families of surfaces and threefolds. As before, \(V\) is a smooth quasi-projective variety.

\begin{ex}[Families of surfaces with Du Val singularities]
\label{ex: surface DV}
Let \(f:U\to V\) be a flat Whitney equisingular family of projective surfaces with Du Val singularities and non-negative Kodaira dimension. Recall that Du Val singularities are isolated rational hypersurface singularities and are classified by the types \(A_n, D_n, E_6, E_7, E_8\).
From the local equations of Du Val singularities, we can easily compute the Milnor sequences for these types as follows:
\begin{center}
\begin{tabular}{ |c|c|c|c| } 
\hline
Types & Milnor sequence \(\mu^{(*)}\) \\
\hline
\(A_n\) & \((n,1,1,1)\) \\ 
\(D_n\) & \((n,2,1,1)\) \\ 
\(E_n\) & \((n,2,1,1)\) \\ 
\hline
\end{tabular}
\end{center}
Considering the miniversal deformation space of Du Val singularities (c.f. \cite{KM98}*{Section 4.3}), the constancy of the Milnor numbers implies the constancy of the type of Du Val singularities. Therefore, \(f\) is Whitney equisingular if and only if each fiber of \(f\) has a fixed number of each type of singularities.

Notice that Du Val singularities are canonically resolved by the series of blow-ups. For example, the family of varieties with one \(A_5\)-singularity on every fiber is simultaneously resolved by the sequence of blow-ups along the critical loci. Consequently, if \(f\) is Whitney equisingular, then \(f\) admits a simultaneous resolution of singularities. Therefore, \(V\) is of log general type if the variation of \(f\) is maximal, directly from the original Viehweg's hyperbolicity conjecture (Theorem \ref{thm: Viehweg's hyperbolicity}).
\end{ex}

For threefolds, Gorenstein terminal singularities (i.e. isolated compound Du Val singularities) are isolated rational hypersurface singularities. Therefore, we may apply Corollary \ref{cor: Viehweg hyperbolicity for isolated hypersurface singularities}. Unlike Du Val surface singularities, the resolution of singularities is very complicated, and thus, we are far from knowing any non-trivial criteria for the existence of a simultaneous resolution.

\begin{ex}[Families of threefolds with isolated compound Du Val singularities]
\label{ex: threefold cDV}
Let \(f:U\to V\) be a flat Whitney equisingular family of projective threefolds with isolated compound Du Val singularities and non-negative Kodaira dimension. Recall that isolated compound Du Val singularities are classified by the types \(cA_n, cD_n, cE_6, cE_7, cE_8\), each type meaning that if we slice the singularity by a general hyperplane, we get a Du Val singularity of the corresponding type. Therefore, these types determine the Milnor numbers \(\mu^{(i)}\) for \(i\le 3\), leaving the Milnor number \(\mu=\mu^{(4)}\) as the only indecisive term in the Milnor sequence. Due to the result of L\^{e}-Ramanujam \cite{LR76}, which states that the constancy of the Milnor number is equivalent to the constancy of the embedded topological type, this verifies the discussion following Corollary \ref{cor: Viehweg hyperbolicity for isolated hypersurface singularities}: the family \(f\) is Whitney equisingular if every fiber of \(f\) has a fixed number of each type of cDV singularities of a fixed embedded topological type.

In particular, for a birationally non-isotrivial family \(g:X\to \P^1\) of threefolds of non-negative Kodaira dimension, whose general fiber has isolated cDV singularities, then there exist at least three fibers that are not equisingular to the general fiber. The following proposition gives a boundary example, whose fibers, except for those over \(\{-2,2,\infty\}\), have exactly one \(cA_5\)-singularity of a fixed embedded topological type.
\end{ex}

\begin{prop}
\label{prop: boundary example}
Let \(f:U\to \P^1\setminus\{-2,2,\infty\}\) be a family of degree \(6\) hypersurfaces in \(\P^4\) parametrized by
\[
\left(w^4t^2-\frac{t^6}{3}+w^4x^2+\frac{x^6}{3}+y^6+\alpha y^3z^3+z^6\right)\subset \P^4\times\P^1\setminus\{-2,2,\infty\}
\]
with coordinates \([t:x:y:z:w]\in \P^4\) and \([\alpha:1]\in \P^1\). Then \(f\) is a flat Whitney equisingular family of canonically polarized varieties with one \(cA_5\)-singularity, and the variation of \(f\) is maximal.
\end{prop}

\begin{proof}
It is easy to check that the fiber \(U_\alpha\subset \P^4\) has an isolated \(cA_5\)-singularity at \([0:0:0:0:1]\in \P^4\). In addition, the Milnor number of the hypersurface singularity \(t^2+x^2+y^6+\alpha y^3z^3+z^6\) at the origin is \(25\), for all \(\alpha\neq-2,2\). See Arnold et al. \cite{Arnold12}*{p.200 Corollary 3} for the computation of the Milnor number of a quasihomogeneous function. Therefore, \(f\) is a flat Whitney equisingular morphism.

It remains to prove that \(f\) is of maximal variation. Since the fibers of \(f\) are canonically polarized, it suffices to show that two general fibers of \(f\) are not isomorphic. By Teissier's ``economy of miniversal deformations" \cite{Teissier77}*{Theorem 4.8.4}, no other small deformation of the isolated hypersurface singularity, \(t^2+x^2+y^6+\alpha y^3z^3+z^6\) at the origin, is analytically isomorphic to the original. Since each fiber of \(f\) has exactly one isolated singularity of this form, two general fibers of \(f\) are not isomorphic to each other, which concludes the proof.
\end{proof}

\end{document}